\documentclass[11pt]{amsart}

\usepackage[margin=1in]{geometry}  
\usepackage{graphicx}              
\usepackage{amsmath}               
\usepackage{amsfonts}    
\usepackage{amssymb}

\newtheorem{thm}{Theorem}[section]
\newtheorem{prop}[thm]{Proposition}
\newtheorem{lem}[thm]{Lemma}
\newtheorem{cor}[thm]{Corollary}

\theoremstyle{definition}

\theoremstyle{remark}
\newtheorem{rem}[thm]{Remark}

\numberwithin{equation}{section}

\newcommand{\R}{\mathbb{R}^n}

\begin{document}

\title{An optimization problem in heat conduction with minimal temperature constraint, interior heating and exterior insulation} 
\author{Hui Yu}
\address{Department of Mathematics, the University of Texas at Austin}
\email{hyu@math.utexas.edu}

\begin{abstract}
We show the existence and optimal regularity of  the optimal temperature configuration in a problem in heat conduction with minimal temperature constraint, interior heating and exterior insulation. Regularity of the two free boundaries is also studied.\end{abstract}

 \maketitle

\tableofcontents

\section{Introduction}
In this paper we discuss an optimization problem in heat conduction that may be briefly described as follows: We want to keep the temperature in a room above a given temperature profile using heating sources inside the room and insulation material of a given volume outside the room. The optimal configuration is the one that takes the least energy.

Mathematically, given a bounded smooth domain $D\subset\R$, a smooth non-negative function $\phi:\R\to\mathbb{R}$ compactly supported in $D$, and a positive number $m>0$, we seek a function $u:\R\to\mathbb{R}$ with $|\{u>0\}\backslash D|=m$ and $u\ge\phi$. Here $|E|$ denotes the Lebesgue measure of a set $E$. We also assume $\Delta u=0$ in $\{u>0\}\backslash D$ due to insulation, and $\Delta u\le 0$ in $D$ due to interior heating.   

Among this class of functions an optimizer should minimize a certain functional corresponding to the energy taken by the interior heating sources. The most natural functional seems to be the total mass of $-\Delta u$ in $D$ $$\int_D -\Delta u dx.$$ However, this functional depends on the shape of $\{u>0\}$ in a highly nonlocal fashion and requires new ideas. Consequently we propose, as a replacement,  to study the Dirichlet energy 
$$\int \frac{|\nabla u|^2}{2} dx.$$ 

These two functionals are of the same order. 

Intuitively, to save energy, one would like to make $u$ as low as possible subject to $u\ge\phi$, and hence $u$ would solve the obstacle problem in $D$ with $\phi$ as the obstacle. Now since $0\le u\le \max{\phi}$ along $\partial D$ and $\phi$ is compactly supported in $D$, one has $c\le\phi\le C$ in the contact set $\{u=\phi\}$ for some $c$ and $C$ depending only on $\phi$ and $D$. Hence with Gauss-Green theorem and the fact that $-\Delta u$ is supported in the contact set, one has the following formal calculation \begin{align*}
\int \frac{|\nabla u|^2}{2} dx &=\int_{\{u>0\}} \frac{|\nabla u|^2}{2} dx\\ &=\int_{\{u>0\}} -u\Delta u/2 dx\\ &=\int_{\{u=\phi\}}-u\Delta udx\\ & \sim \int_{\{u=\phi\}}-\Delta udx\\&=\int_D -\Delta udx.
\end{align*}

As a result, we propose to study the following optimization problem:

\textbf{Physical Problem:} Find a minimizer of the Dirichlet energy $$\int \frac{|\nabla u|^2}{2} dx$$ over $K_0=\{u\in H^{1}_{0}(\R):u\ge \phi,  |\{u>0\}\backslash D|=m,  \Delta u\le 0 \text{ in } D, \Delta u=0 \text{ in } \{u>0\}\backslash D\}.$ 

Here the inequalities on $\Delta u$ are understood in the distributional sense.

Concerning the minimizer our main result is the existence and optimal regularity \begin{thm}
There exists a minimizer to the Physical Problem. This minimizer is Lipschitz continuous in $\R$.
\end{thm} 

There are two free boundaries coming from the interior contact set and the exterior boundary  $\partial\{u>0\}$.
These correspond to the boundary of effective heating sources and the boundary of insulation material, respectively. Concerning the regularity of the interior free boundary we establish

\begin{thm}
For  $\Delta \phi$ uniformly negative in $\{\phi>0\}$, the interior free boundary $\partial (\{u>\phi\}\cap D)$ is smooth except on a set of singular points, which are covered by a countable union of lower-dimensional $C^1$ manifolds.
\end{thm} 

Concerning the exterior free boundary we have \begin{thm}
The exterior free boundary $\partial \{u>0\}$ is smooth except on a $H^{n-1}$-null set.
\end{thm} Here $H^{n-1}$ is the $(n-1)$-dimensional Hausdorff measure.

Similar problems have been studied by Alt-Caffarelli \cite{AC}, Aguilera-Alt-Caffarelli \cite{AAC}, Aguilera-Caffarelli-Spruck \cite{ACS} and Teixeira \cite{T}, where the authors studied various functionals that are of the same order of the Dirichlet energy. The results and techniques in this paper are very much inspired by these previous work. However there are also significant differences. 

On the physical level, instead of prescribing temperature along the walls of the room as in previous works, we consider a minimal temperature profile in the interior of the room. This changes the problem from a boundary value problem in $D^c$ to a problem in the entire $\R$. 

This leads to some new difficulties, the most fundamental one being the sign-changing $\Delta u$. In all previous works, $u$ is a subsolution throughout the domain of concern, which is the source of regularity of the minimizer. Here, however, $\Delta u$ changes signs. We get around this by studying a series of perturbed problems. These perturbed problems obtain very regular solutions that converge to a minimizer of our problem. Also Lipschitz regularity is persistent along this limiting process, which gives the optimal regularity of the minimizer.

Once the optimal regularity of the minimizer is established, the problem naturally splits into an interior obstacle problem and an exterior one-phase problem. This allows us to use previous results and hence establish the regularity of the two free boundaries.

This paper is organized as follows: in Section 2 we introduce a three-parameter-family of perturbed problems.  We show the existence and first properties of minimizers to these problems. In Section 3 we give estimates uniform in one of the parameters. This is exploited in Section 4 to obtain `asymptotic' minimizers of a two-parameter family of perturbed problems. Estimates uniform in one of remaining two parameters is also established in Section 4. This gives rise to limiting solutions to yet another family of perturbed problems, which now depend on one parameter.  In Section 5. we study the free boundary regularity of these limiting solutions.   In the last section, Section 6, we connect this one-parameter perturbed problem to our original Physical Problem by showing that when the last parameter is small enough, a minimizer of this perturbed problem actually solves our Physical Problem. This completes the proof for main results as estimates on minimizers of the perturbed problems apply to a minimizer to the Physical Problem. We also show that the positive phases of these minimizers are well-localized in a bounded set, hence any local estimates is actually uniform over the domain, and that the optimization in $\R$ is the same as in a big but bounded set.

\section{A three-parameter family of perturbed problems}
For small positive parameters $\kappa_1$, $\kappa_2$ and $\epsilon$, we define the following functions:\begin{itemize}\item$A_{\kappa_1}:\mathbb{R}\to\mathbb{R}$ is a \underline{nonnegative} \underline{decreasing} \underline{convex}  function that \underline{vanishes on $[0,+\infty)$}. It equals $-\frac{1}{\kappa_1}(t-\frac{\kappa_1}{2})$ for $t<-\kappa_1$, and smoothly interpolates between $-\kappa_1$ and $0$.

$\alpha_{\kappa_1}$ is the derivative of $A_{\kappa_1}$.
\item $B_{\kappa_2}:\mathbb{R}\to\mathbb{R}$ is a \underline{piecewise linear} function that \underline{vanishes on $(-\infty,0]$ }  and  \underline{ equals $1$ on $[\kappa_2,+\infty)$}.

$\beta_{\kappa_2}$ is the derivative of $B_{\kappa_2}$

\item $f_\epsilon:\mathbb{R}\to\mathbb{R}$ is the \underline{piecewise linear} function that \underline{equals $0$ at $m$}, has \underline{slope $\frac{1}{\epsilon}$ to the right of $m$}, and \underline{slope $\epsilon$ to the left of $m$}. 
\end{itemize} We study the following three-parameter functional \begin{equation}
J_{\kappa_1,\kappa_2,\epsilon}(w)=\int\frac{|\nabla w|^2}{2}dx+A_{\kappa_1}(w-\phi)+f_{\epsilon}(\int_{D^c}B_{\kappa_2}(u))
\end{equation} over all $H^1_0(\R)$ functions.

\begin{rem}
We enlarge the class of functions under consideration from $K_0$ to all $H^1_0(B_R)$ functions. To obtain solutions to our original problem we impose three-parameter penalization/ regularization. $A_{\kappa_1}$ is to penalize functions that do not lie above $\phi$, which seems a standard technique in the study of obstacle-type problems \cite{PSU}. $B_{\kappa_2}$ is to regularize  $u\mapsto |\{u>0\}\backslash D|$ as in Caffarelli-Salsa \cite{CS}. $f_\epsilon$ is to penalize  functions with the wrong volume of positive phase \cite{AAC}.
\end{rem} 

\begin{rem}
We would often suppress subscripts when there is no ambiguity.
\end{rem}

The following gives a competitor that may be far from optimal, but it is universal in the sense that it gives estimates independent of all parameters.
\begin{prop}
There is $M=M(\phi, D)$ and $w\in K_0$ such that $$J_{\kappa_1,\kappa_2,\epsilon}(w)=M$$ for all positive $\kappa_1,\kappa_2$ and $\epsilon$.
\end{prop} 

\begin{proof}
Let $w$ be the minimizer of the Dirichlet energy among all $H^1_0(D)$ functions above $\phi$. Then $A(w-\phi)$ is constantly $0$. $B_{\kappa_2}(w)$ vanishes outside $D$ thus $$J_{\kappa_1,\kappa_2,\epsilon}(w)=\int\frac{|\nabla w|^2}{2}dx=:M$$ independent of the parameters.

Obviously $w\ge \phi$. $\Delta w\le 0$ in $D$ as a standard result from obstacle problem. $\{w>0\}\backslash D=\emptyset$. Thus $w\in K_0$.
\end{proof} 

Next we establish the existence of minimizers to the perturbed problems.

\begin{prop}
There exist minimizers to the perturbed functionals.
\end{prop} 

\begin{proof}
The function in Proposition 2.3 is a competitor and shows the functional is not always infinite. 

Then the existence follows from a standard argument using the direct method \cite{AC}.
\end{proof} 

The next proposition gives an $\mathcal{L}^{\infty}$ estimate on minimizers, and shows we only need to consider functions with one phase.
\begin{prop}
If $u_{\kappa_1,\kappa_2,\epsilon}$ is a mimimizer to $J_{\kappa_1,\kappa_2,\epsilon}$, then $$0\le u_{\kappa_1,\kappa_2,\epsilon}\le \max{\phi}.$$
\end{prop}

\begin{proof}
Use $u_{\kappa_1,\kappa_2,\epsilon}-t\min(u,0)$ and $u_{\kappa_1,\kappa_2,\epsilon}-t\min(u-\max{\phi},0)$ as competitors and study the first order behavior as $t\to 0^{+}$.
\end{proof} 

\begin{prop}
If $u_{\kappa_1,\kappa_2,\epsilon}$ is a mimimizer to $J_{\kappa_1,\kappa_2,\epsilon}$, then \begin{equation}
\Delta u_{\kappa_1,\kappa_2,\epsilon}=\alpha_{\kappa_1}(u_{\kappa_1,\kappa_2,\epsilon}-\phi)+f'(\int_{D^c}B_{\kappa_2}(u_{\kappa_1,\kappa_2,\epsilon}))\beta_{\kappa_2}(u_{\kappa_1,\kappa_2,\epsilon})\chi_{D^c}.
\end{equation} 
\end{prop} 

\begin{proof}
This is the Euler-Lagrange equation of the perturbed energy functional.
\end{proof} 

\section{Sending $\kappa_1\to 0$}
In this section we give uniform $C^{1,\alpha}$-estimate of minimizers independent of $\kappa_1$, establishing compactness when $\kappa_1\to 0$. The limiting function lies above $\phi$ and asymptotically minimizes a two-parameter functional.

With standard regularity theory for elliptic equations, Proposition 2.6 gives $C^{1,\alpha}$ and $W^{2,p}$ regularity of the minimizers for any $0<\alpha<1$ and $1\le p<+\infty$. The goal, however, is to establish estimates independent of $\kappa_1$. This begins with an uniform estimates on $\|\alpha_{\kappa_1}(u-\phi)\|_{\infty}$.

\begin{prop}
$\|\alpha_{\kappa_1}(u-\phi)\|_{\infty}\le \|\phi\|_{C^{1,1}}+\frac{1}{\epsilon\kappa_2}.$
\end{prop} 

\begin{proof}
Define $\tilde{u}=u-\phi$, then the equation for $\tilde{u}$ is $$\Delta \tilde{u}=\alpha_{\kappa_1}(\tilde{u})+f'(\int_{D^c}B(u))\beta(u)\chi_{D^c}-\Delta\phi.$$

Since for each fixed $\kappa_1>0$ $\alpha_{\kappa_1}$ is a bounded smooth function, $\alpha(\tilde{u})^p$ can be used as a test function for this equation:

\begin{equation*}
0=\int \nabla\tilde{u}\cdot (p\alpha(\tilde{u})^{p-1}\alpha'(\tilde{u})\nabla\tilde{u})+\alpha(\tilde{u})^{p+1}+f'(\int_{D^c}B(u))\beta(u)\chi_{D^c}\alpha(\tilde{u})^p-\Delta\phi\alpha(\tilde{u})^p.
\end{equation*}
If we choose $p$ to be even, then the first two terms are negative due to monotonicity and convexity of $A$. Consequently one has \begin{align*}\int_D|\alpha(\tilde{u})|^{p+1}&\le \int_D f'(\int_{D^c}B(u))\beta(u)\chi_{D^c}\alpha(\tilde{u})^p-\Delta\phi\alpha(\tilde{u})^p\\&\le (\int_D |f'(\int_{D^c}B(u))\beta(u)+\Delta\phi|^p)^{1/p}(\int_D|\alpha(\tilde{u})|^{p+1})^{\frac{p}{p+1}}.\end{align*} Note that we used the fact that $\alpha(\tilde{u})$ is supported in $D$ since $u\ge\phi$ outside $D$. 

As a result, $\|\alpha(\tilde{u})\|_{\mathcal{L}^{p+1}(D)}\le(\|\phi\|_{C^{1,1}}+\frac{1}{\epsilon\kappa_2})|D|^{1/p}.$ Normalizing the Lebesgue measure gives

$$(\int_D|\alpha(\tilde{u})|^{p+1}\frac{dx}{|D|})^{\frac{1}{p+1}}\le (\|\phi\|_{C^{1,1}}+\frac{1}{\epsilon\kappa_2})|D|^{\frac{1}{p}-\frac{1}{p+1}}.$$

$p\to +\infty$ gives the desired estimate.

\end{proof}

The proposition above says the right-hand side of (2.2) is bounded independent of $\kappa_1$, which gives uniform $C^{1,\alpha}$-estimate of minimizers.

\begin{thm}
Let $u_{\kappa_1,\kappa_2,\epsilon}$ be a minimizer of $J_{\kappa_1,\kappa_2,\epsilon}$, then for any compact set $K$ and $0<\alpha<1$ one has $$\|u_{\kappa_1,\kappa_2,\epsilon}\|_{C^{1,\alpha}(K)}\le C(K,\alpha, n)(\|\phi\|_{C^{1,1}}+\frac{1}{\epsilon\kappa_2}).$$
\end{thm} 

The following is a direct consequence via Arzela-Ascoli:
\begin{cor}
Up to a subsequence $\kappa_1\to 0$, $u_{\kappa_1}$ converges to some $u$ weakly in $H^{1}_{0}(\R)$ and locally uniformly in $C^{1,\alpha}(\R)$.\end{cor}

With the uniform bound on energy as in Proposition 2.3, the limit $u$ lies above our obstacle:
\begin{prop}
$u\ge\phi$.
\end{prop} 

\begin{proof}
For given $\delta>0$ and compact $K$, $\{u-\phi<-\delta\}\cap K$ is contained in $\{u_{\kappa_1}-\phi<-\delta/2\}\cap K$ for small $\kappa_1$.  For the latter set one has the following estimate
$$M\ge\int A_{\kappa_1}(u_{\kappa_1}-\phi)dx\ge \frac{1}{\kappa_1}\delta/2|\{u_{\kappa_1}-\phi<-\delta/2\}\cap K|.$$ Taking $\kappa_1\to 0$ forces $\{u_{\kappa_1}-\phi<-\delta/2\}\cap K$ to be null.
\end{proof}

\section{Sending $\kappa_2\to 0$}
Now we define a new two-parameter family of perturbed functionals that do not involve $\kappa_1$ anymore:\begin{equation}
J_{\kappa_2,\epsilon}(w)=\int\frac{|\nabla w|^2}{2}dx+f_{\epsilon}(\int_{D^c}B_{\kappa_2}(w)).
\end{equation} 

Ideally we would expect the limit $u$ from Corollary 3.3 to be a minimizer to this new functional over functions that lie above $\phi$. However this is not always true due to the lack of convexity in $u\mapsto f_{\epsilon}(\int_{D^c}B_{\kappa_2}(u)).$ Nevertheless we show that $u$ minimizes the energy `asymptotically' as in the next lemma. It is a variation of the classical lemma of Minty \cite{KS} applied to an operator with a monotone part $u\mapsto\Delta u$ and a regular part $ u\mapsto\beta(u)$.

\begin{lem}
Let $u$ be as in Corollary 3.3, then for any $v\in H^{1}_0(\R)$ with $v\ge\phi$, one has \begin{equation}
\frac{d}{d\lambda}\Big|_{\lambda=0^{+}}J_{\kappa_2,\epsilon}(u+\lambda(v-u))\ge 0.
\end{equation} 
\end{lem} 

\begin{proof}
Since $u_{\kappa_1}$ is a minimizer of $J_{\kappa_1,\kappa_2,\epsilon}$, for any $v\in H^{1}_0$  and $t>0$ one has 
\begin{align*}\int \frac{|\nabla u_{\kappa_1}+t\nabla (v-u_{\kappa_1})|^2}{2}+A(u_{\kappa_1}+t(v-u_{\kappa_1})-\phi)+&f(\int_{D^c}B(u_{\kappa_1}+t(v-u_{\kappa_1})))\\&\ge \int \frac{|\nabla u_{\kappa_1}|^2}{2}+A(u_{\kappa_1}-\phi)+f(\int_{D^c}B(u_{\kappa_1})).\end{align*}

Thus one has a sign on the first order term
$$\int\nabla u_{\kappa_1}\cdot\nabla(v-u_{\kappa_1})+\alpha(u_{\kappa_1}-\phi)(v-u_{\kappa_1})+f'(\int_{D^c}B(u_{\kappa_1}))\int_{D^c}\beta(u_{\kappa_1})(v-u_{\kappa_1})\ge 0.$$

Note that we have the following monotonicity for the first two terms of the operator, coming from the monotonicity of the Dirichlet energy and the function $\alpha$:
$$\int(\nabla v-\nabla u_{\kappa_1})\cdot\nabla(v-u_{\kappa_1})+(\alpha(v-\phi)-\alpha(u_{\kappa_1}-\phi))(v-u_{\kappa_1})\ge 0.$$

Combining the two inequalities above we have for any $v\in H^{1}_0$

$$\int\nabla v\cdot\nabla(v-u_{\kappa_1})+\alpha(v-\phi)(v-u_{\kappa_1})+f'(\int_{D^c}B(u_{\kappa_1}))\int_{D^c}\beta(u_{\kappa_1})(v-u_{\kappa_1})\ge 0.$$ And in particular for $v\ge\phi$

$$\int\nabla v\cdot\nabla(v-u_{\kappa_1})+f'(\int_{D^c}B(u_{\kappa_1}))\int_{D^c}\beta(u_{\kappa_1})(v-u_{\kappa_1})\ge 0.$$

Due to weak convergence in $H^1_0$ of $u_{\kappa_1}\to u$, $$\int\nabla v\cdot\nabla(v-u_{\kappa_1})\to \int\nabla v\cdot\nabla(v-u).$$

The rest of the terms are more regular and we have the following 
\begin{align*}|\int_{D^c}\beta(u_{\kappa_1})(v-u_{\kappa_1})-\int_{D^c}\beta(u)(v-u)|&\le |\int_{D^c}\beta(u_{\kappa_1})(u_{\kappa_1}-u)|+|\int_{D^c}(\beta(u_{\kappa_1})-\beta(u))(v-u)|\\&\le C(\kappa_2)\|u_{\kappa_1}-u\|_{\mathcal{L}^2}+|\int_{D^c}(\beta(u_{\kappa_1})-\beta(u))(v-u)|\\&=o(1)+|\int_{D^c}(\beta(u_{\kappa_1})-\beta(u))(v-u)|.
\end{align*}

It remains to show $|\int_{D^c}(\beta(u_{\kappa_1})-\beta(u))(v-u)|\to 0$. To this end, note that $(\beta(u_{\kappa_1})-\beta(u))$ is bounded and $u-v\in\mathcal{L}^2$, thus for any given $\delta>0$ we can find $R$ big enough so that $$|\int_{D^c}(\beta(u_{\kappa_1})-\beta(u))(v-u)-\int_{D^c\cap B_R}(\beta(u_{\kappa_1})-\beta(u))(v-u)|\le\delta.$$ Now on the compact set $B_R$ one can apply bounded convergence theorem to show $$\int_{D^c\cap B_R}(\beta(u_{\kappa_1})-\beta(u))(v-u)\to 0.$$

As a result one has the desired estimate $$\int\nabla v\cdot\nabla(v-u)+f'(\int_{D^c}B(u))\int_{D^c}\beta(u)(v-u)\ge 0.$$
\end{proof}

Let $u_{\kappa_1,\kappa_2,\epsilon}$ be a minimizer of $J_{\kappa_1,\kappa_2,\epsilon}$ and $u_{\kappa_2}$ be the limit when $\kappa_1\to 0$ as in Corollary 3.3, we now begin the program of sending $\kappa_2\to 0$. For this one needs estimate on $u_{\kappa_2}$ uniform in $\kappa_2$. The previous lemma gives the equation for $u_{\kappa_2}$:

\begin{cor}
\begin{equation}\Delta u=f'(\int_{D^c}B(u))\beta(u)\chi_{D^c} \text{ in $\{u>\phi\}$}.\end{equation}

\begin{equation}-(\|\phi\|_{C^{1,1}}+\frac{1}{\epsilon\kappa_2})\le\Delta u\le f'(\int_{D^c}B(u))\beta(u)\chi_{D^c} \text{ in $\R$}.\end{equation}
\end{cor} 

\begin{proof}
Equation (4.3) and the right-hand side of (4.4) are direct consequence of the previous lemma. 

The left-hand side of (4.4) comes from the weak convergence of $u_{\kappa_1}\to u$ in $H^1_0$ and the uniform bound on the right-hand side of (2.2).
\end{proof} 

The domain naturally splits into three regions: $\{u\le\kappa_2\}$, $\{u>\kappa_2\}\cap\{u>\phi\}$ and $\{u=\phi\}$. In the first region we have smallness of data, in the second $u$ is harmonic, and in the last $u$ induces regularity from the obstacle. The following propositions establish estimates in these regions.

\begin{prop}
If $x_0\in\{u\le\kappa_2\}$, then $$|\nabla u(x_0)|\le C(n)(\|\phi\|_{C^{1,1}}\kappa_2+\frac{1}{\epsilon}+1).$$
\end{prop} 

\begin{proof}
Define $$w(y)=\frac{1}{\kappa_2}u(x_0+\kappa_2y).$$ Then $$-(\|\phi\|_{C^{1,1}}\kappa_2+\frac{1}{\epsilon})\le\Delta w\le \frac{1}{\epsilon}.$$ Also $w(0)\le 1$.

Being a nonnegative function with bounded Laplacian, $|w|\le C(n)(\|\phi\|_{C^{1,1}}\kappa_2+\frac{1}{\epsilon}+1)$ in $B_1$. Thus $|\nabla w(0)|\le C(n) (\|\phi\|_{C^{1,1}}\kappa_2+\frac{1}{\epsilon}+1)$ by standard interior estimates for elliptic equations.

Note $\nabla w(0)=\nabla u(x_0)$ one sees the desired estimate.
\end{proof} 

\begin{prop}
If $x_0\in\{u=\phi\}$, then $$|\nabla u(x_0)|\le C(n)(\|\phi\|_{C^1}+\|\phi\|_{C^{1,1}}\kappa_2+\frac{1}{\epsilon}+1).$$
\end{prop} 

\begin{proof}
Define  $$w(y)=\frac{1}{\kappa_2}(u(x_0+\kappa_2y)-\phi(x_0+\kappa_2y)).$$ Then $w\ge 0$ and $w(0)=0$. Moreover, $$-(2\|\phi\|_{C^{1,1}}\kappa_2+\frac{1}{\epsilon})\le\Delta w\le (\|\phi\|_{C^{1,1}}\kappa_2+\frac{1}{\epsilon}).$$

Note $\nabla u(x_0)=\nabla w(0)+\nabla\phi(x_0)$, the estimate follows from elliptic regularity as in the previous proposition.
\end{proof} 

\begin{prop}
If $x_0\in\{u>\phi\}\cap\{u>\kappa_2\}$, then $$|\nabla u(x_0)|\le C(n)(\|\phi\|_{C^1}+\|\phi\|_{C^{1,1}}\kappa_2+\frac{1}{\epsilon}).$$
\end{prop} 

\begin{proof}
Let $d=dist(x_0,\partial(\{u>\phi\}\cap\{u>\kappa_2\})),$ and $y_0\in \partial(\{u>\phi\}\cap\{u>\kappa_2\})$ be such that $|y_0-x_0|=d$. Then in particular $u(y_0)=\kappa_2$ or $u(y_0)=\phi(y_0)$.

If $u(y_0)=\kappa_2$, we define $$w(y)=\frac{1}{d}(u(x_0+dy)-\kappa_2).$$ Then one has $$\Delta w=0 \text{ in $B_1$},$$ $$w\ge 0 \text{ in $B_1$},$$ $$w(\tilde{y_0})=0$$ and $$|\nabla w(\tilde{y_0})|\le  C(n)(\|\phi\|_{C^{1,1}}\kappa_2+\frac{1}{\epsilon}+1).$$ Here $\tilde{y_0}$ is the point on $\partial B_1$ corresponding to $y_0$. The last estimate comes from Proposition 4.3.

By Harnack, $w(y)\ge c(n)w(0)$ in $B_{1/2}$. Define a scaled fundamental solution $$\Psi(y)=\frac{cw(0)}{2^{n-2}-1}(\frac{1}{|y|^{n-2}}-1),$$ then $\Psi=0$ on $\partial B_1$, $\Psi=cw(0)$ along $\partial B_{1/2}$ and $\Delta\Psi=0$ in $B_1\backslash B_{1/2}$. Comparison principle then gives $\Psi\le w$ in $B_1\backslash B_{1/2}$. 

However, $w(\tilde{y_0})=\Psi(\tilde{y_0})$ thus $\nabla w(\tilde{y_0})\cdot n\ge \nabla\Psi(\tilde{y_0})\cdot n$, where $n$ is the inner normal vector to $B_1$ at $\tilde{y_0}$.

This gives $$C(n)(\|\phi\|_{C^{1,1}}\kappa_2+\frac{1}{\epsilon}+1)\ge|\nabla w(\tilde{y_0})|\ge\nabla\Psi(\tilde{y_0})\cdot n=c(n)w(0).$$ Then $w(y)\le C(n)w(0)\le C(n)(\|\phi\|_{C^{1,1}}\kappa_2+\frac{1}{\epsilon}+1)$ in $B_{1/2}$ by Harnack, and elliptic regularity gives $$|\nabla u(x_0)|=|\nabla w(0)|\le C(n)(\|\phi\|_{C^{1,1}}\kappa_2+\frac{1}{\epsilon}+1).$$

For the case when $y_0\in \{u=\phi\}$, we define $w(y)=\frac{1}{d}(u(x_0+dy)-\phi(x_0+dy))$, which is another nonnegative harmonic function in $B_1$ that vanishes at one point on $\partial B_1$ where we have a gradient estimate from Proposition 4.4. Thus similar barrier argument applies. 
\end{proof} 

We collect these results to get the uniform Lipschitz estimate independent of $\kappa_2$:

\begin{thm}
Let $u_{\kappa_2}$ be the limit as in Corollary 3.3 of $u_{\kappa_1,\kappa_2}$ as $\kappa_1\to 0$.  Then $$|\nabla u_{\kappa_2}|\le C(n)(\|\phi\|_{C^1}+\|\phi\|_{C^{1,1}}\kappa_2+\frac{1}{\epsilon}).$$
\end{thm} 
\begin{rem}
Since there is a jump in gradient along $\partial\{u>0\}$, this Lipschitz estimate is the optimal regularity. See \cite{AC}.
\end{rem} 
Again by Arzela-Ascoli we have the following 
\begin{cor}
Let $u_{\kappa_1,\kappa_2,\epsilon}$ be a minimizer of $J_{\kappa_1,\kappa_2,\epsilon}$. Then up to a subsequence as $\kappa_1,\kappa_2\to 0$, $u_{\kappa_1,\kappa_2,\epsilon}\to u_\epsilon$ weakly in $H^1_0$ and locally uniformly in $C^{\alpha}$ for any $0<\alpha<1$. Moreover $$|\nabla u_{\epsilon}|\le C(n)(\|\phi\|_{C^1}+1/\epsilon).$$ \end{cor}

\section{Regularity of free boundaries}
Now we can define a family of perturbed functionals with only one parameter: \begin{equation}
J_{\epsilon}(w)=\int\frac{|\nabla w|^2}{2}+f_{\epsilon}(|\{w>0\}\backslash D|).
\end{equation} 

Our first result is the minimality of $u$:

\begin{thm}
Let $u_\epsilon$ be as in Corollary 4.8. Then it is a local minimizer of $J_{\epsilon}$ over functions above $\phi$. 
\end{thm} 

\begin{proof}
Suppose, on the contrary, that there is $\delta>0$, and $v\ge\phi$ with $v-u$ supported in $B_r(x_0)$ such that \begin{equation}\int_{B_r(x_0)}\frac{|\nabla v|^2}{2}+f(\int_{D^c\cap B_r(x_0)}\chi_{\{v>0\}})<\int_{B_r(x_0)}\frac{|\nabla u|^2}{2}+f(\int_{D^c\cap B_r(x_0)}\chi_{\{u>0\}})-\delta.\end{equation}

Since $B_{\kappa_2}(v)\to \chi_{\{v>0\}}$ as $\kappa_2\to 0$, the left-hand side of (5.2) satisfies the following \begin{align*}\int_{B_r(x_0)}\frac{|\nabla v|^2}{2}+f(\int_{D^c\cap B_r(x_0)}\chi_{\{v>0\}})&=\lim_{\kappa_2\to 0}\int_{B_r(x_0)}\frac{|\nabla v|^2}{2}+f(\int_{D^c\cap B_r(x_0)}B_{\kappa_2}(v))\\&=\lim_{\kappa_1,\kappa_2\to 0}\int_{B_r(x_0)}\frac{|\nabla v|^2}{2}+A_{\kappa_1}(v-\phi)+f(\int_{D^c\cap B_r(x_0)}B_{\kappa_2}(v)).\end{align*}

Meanwhile, if we fix a small $\gamma>0$, then the right-hand side of (5.2) satisfies
\begin{align*}
\int_{B_r(x_0)}\frac{|\nabla u|^2}{2}+f(\int_{D^c\cap B_r(x_0)}\chi_{\{u>0\}})&-\delta\le\int_{B_r(x_0)}\frac{|\nabla u|^2}{2}+f(\int_{D^c\cap B_r(x_0)\cap\{u\ge\gamma\}}\chi_{\{u>0\}})-\delta/2\\&\le\underline{\lim}_{\kappa_2\to 0}\int_{B_r(x_0)}\frac{|\nabla u_{\kappa_2}|^2}{2}+f(\int_{D^c\cap B_r(x_0)\cap\{u\ge\gamma\}}\chi_{\{u_{\kappa_2}>0\}})-\delta/2\\&\le\underline{\lim}_{\kappa_2\to 0}\int_{B_r(x_0)}\frac{|\nabla u_{\kappa_2}|^2}{2}+f(\int_{D^c\cap B_r(x_0)\cap\{u\ge\gamma\}}B_{\kappa_2}(u_{\kappa_2}))-\delta/2\\&\le \underline{\lim}_{\kappa_1, \kappa_2\to 0}\int_{B_r(x_0)}\frac{|\nabla u_{\kappa_1,\kappa_2}|^2}{2}+f(\int_{D^c\cap B_r(x_0)\cap\{u\ge\gamma\}}B_{\kappa_2}(u_{\kappa_1,\kappa_2}))-\delta/2\\&\le\underline{\lim}_{\kappa_1,\kappa_2\to 0}\int_{B_r(x_0)}\frac{|\nabla u_{\kappa_1,\kappa_2}|^2}{2}+A_{\kappa_1}(u_{\kappa_1,\kappa_2}-\phi)\\&+f(\int_{D^c\cap B_r(x_0)}B_{\kappa_2}(u_{\kappa_1,\kappa_2}))-\delta/2.
\end{align*} Here we used Fatou's lemma and the fact that $B_{\kappa_2}(u_{\kappa_2})=\chi_{\{u_{\kappa_2>0}\}}$ on $\{u>\gamma\}$ as long as $\kappa_2$ is small enough.

Combining these inequalities we could conclude that $$J_{\kappa_1,\kappa_2}(v)<J_{\kappa_1,\kappa_2}(u_{\kappa_1,\kappa_2})-\delta/4$$ for small $\kappa_1,\kappa_2$, contradicting the minimality of $u_{\kappa_1,\kappa_2}$.
\end{proof} 

As a simple corollary we have the Euler-Lagrange equation satisfied by $u$:
\begin{cor}
$$\Delta u\le 0 \text{ in $D$}.$$
$$\Delta u=0 \text{ in $D\cap\{u>\phi\}$}.$$
$$\Delta u\ge 0 \text{ in $D^c$}.$$
$$\Delta u=0 \text{ in $\{u>0\}\backslash D$}.$$
\end{cor}

Also $u$ minimizes the Dirichlet energy over $K_1:=\{w\in H^{1}(D)|w\ge\phi, w=u \text{ on $\partial D$}\}$, thus $u$ is an obstacle solution in $D$ with $\phi$ as obstacle and $u\big|_{\partial D}$ as boundary data. Therefore the standard theory of obstacle problem applies and gives the regularity of the interior free boundary $\partial(\{u>\phi\}\cap D)$ \cite{PSU}:

\begin{thm}
For $\Delta\phi$ uniformly negative in $\{\phi>0\}$, the interior free boundary $\partial(\{u>\phi\}\cap D)$ is smooth except on a set of singular points, which are covered by a countable union of lower dimensional $C^1$-manifolds.
\end{thm} 

Regularity of the exterior free boundary $\partial\{u>0\}$ begins with the following non-degeneracy lemma, which can be proved with the same techniques as in Lemma 3.4 of \cite{AC}:
\begin{lem}
There is $c=c(n)$ such that $u=0$ in $B_{R/2}(x_0)$ whenever $$\frac{1}{R\cdot H^{n-1}(\partial B_R)}\int_{\partial B_R(x_0)} udH^{n-1}< c(n)\epsilon$$ and $B_R(x_0)\subset (D\cap\{u=\phi\})^c$. 
\end{lem} 

This along with the uniform Lipschitz estimate gives the following lower density estimate of the positive phase: 

\begin{lem}
For $x_0\in\overline{\{u>0\}}$ and $B_R(x_0)\subset  (D\cap\{u=\phi\})^c$, then $$\frac{|B_R\cap\{u>0\}|}{|B_R|}\ge \frac{c(n)}{\|\phi\|_{C^1}+1/\epsilon}\epsilon.$$
\end{lem} 

\begin{proof}
By the previous lemma there is $y_0\in\partial B_{R/2}(x_0)$ such that $u(y_0)>c(n)\epsilon R$. By Lipschitz continuity $u(y)>0$ if $|y-y_0|\le c(n)\epsilon R/Lip(u).$
\end{proof} 

Note that as long as we do not touch the interior contact set, we have all ingredients for the theory of harmonic functions with linear growth as in \cite{AC}. Consequently we can use Theorem 4.5  and 4.8 there to obtain the following structure theorem:
\begin{thm}
$H^{n-1}(K\cap\partial\{u>0\})<\infty$ for any compact $K$.

There is a Borel $q_u$ such that $\Delta u\big|_{(D\cap\{u=\phi\})^c}=q_uH^{n-1}\big|_{\partial\{u>0\}}$.

For any compact set $K$ there are $0<c(n,\|\phi\|_{C^1},K,\epsilon)\le C(n,\|\phi\|_{C^1},K,\epsilon)<\infty$ such that for $x_0\in\partial\{u>0\}$ and $B_r(x_0)\subset(D\cap\{u=\phi\})^c$ one has $$c\le q_u(x_0)\le C$$ and $$cr^{n-1}\le H^{n-1}(B_r(x_0)\cap\partial\{u>0\})\le Cr^{n-1}.$$

For $H^{n-1}$-almost every $x_0$ in $\partial\{u>0\}$, $$u(x_0+x)=q_u(x_0)\max\{-x\cdot\nu(x_0),0\}+o(|x|)$$ where $\nu(x_0)$ is the outer normal to the reduced boundary of $\{u>0\}$ at $x_0$.
\end{thm} 

Moreover with techniques from \cite{AAC} one see the following:

\begin{thm}
$q_u$ is constant $H^{n-1}$ almost everywhere on $\partial\{u>0\}$. 
\end{thm} 

From here the theory of weak solutions in \cite{AC} can be applied to obtain the following:

\begin{thm}
$\partial\{u>0\}$ is smooth except on a $H^{n-1}$-null set. 
\end{thm}

\section{Connection to the Physical Problem}

In this section we show that for small $\epsilon$ a solution to the one-parameter perturbed problem actually solves the original Physical Problem.  To this end we first have the following:

\begin{prop}
For small $\epsilon>0$, $|\{u>0\}\backslash D|=m$. 
\end{prop} 

\begin{proof}
Note that in the exterior domain $D^c$ our solution solves the problem in \cite{AAC} with a Lipschitz boundary datum. Hence Theorem 7 there can be applied.
\end{proof} 
\begin{rem}
In particular we do not need to send $\epsilon\to 0$.
\end{rem} 
We collect results on $u_\epsilon$ for small $\epsilon$ as in the previous proposition to obtain the following:

\begin{thm}
For $\epsilon>0$ small, $u_\epsilon$ optimizes the Physical Problem.
\end{thm} 

\begin{proof}
Since $|\{u>0\}\backslash D|=m$, $f(|\{u>0\}\backslash D|)$ vanishes. Thus $u$ is a minimizer for the Dirichlet energy over functions above $\phi$.

Also Corollary 5.2 establishes right signs on the Laplacian of $u$. Hence $u\in K_0$.
\end{proof} 

The next theorem states that positive phase is well localized inside a bounded set. As a result, outside the interior contact set, any local estimate can be upgraded to global estimate with constants independent of the compact set.  Also, the optimization in $\R$ is actually the same as in a big but bounded set. 

\begin{thm}
Let $u$ be a minimizer, then $$diam(\{u>0\})\le diam(D)+1+C(n)m\frac{\|\phi\|_{C^1}+1/\epsilon}{\epsilon}.$$
\end{thm} 

\begin{proof}
For $x\in\overline{\{u>0\}}$ and $dist(x,D)\ge 1$, $B_1(x)\cap D=\emptyset$ and consequently $$|\{u>0\}\cap B_1(x)|\ge \frac{C(n)\epsilon}{\|\phi\|_{C^1}+1/\epsilon}.$$

By Vitali we reduce $\{B_1(x)\}_{x\in\overline{\{u>0\}} \text{ and } dist(x,D)\ge 1}$ to a disjoint subcollection $\{B_1(x_j)\}_{j\in J}$ with  $\{B_5(x_j)\}_{j\in J}$ still covers $\overline{\{u>0\}}\cap \{dist(x,D)\ge 1\}$.

Thus one has the following 
\begin{align*}m&=|\{u>0\}\cap D^c|\\&\ge\Sigma_J|B_1(x_j)\cap\{u>0\}|\\&\ge\frac{C(n)\epsilon}{\|\phi\|_{C^1}+1/\epsilon}Card(J).
\end{align*}

As a result we have estimate on the cardinality of $J$. Note that any $x\in\{u>0\}$ can be connected to $\{dist(x,D^c)\le 1\}$ through a chain of at most $Card(J)$ balls of radius 5, we have the desired estimate.

\end{proof}

\section*{Acknowledgement} The author would like to thank his PhD advisor, Luis Caffarelli, for many valuable conversations regarding this project. He is also grateful to his colleagues and friends, especially Luis Duque, Dennis Kriventsov and Yijing Wu,  for all the discussions and encouragement. 



\begin{thebibliography}{100}
\bibitem{AAC}N. Aguilera, H. W. Alt,  L. Caffarelli, {\em An optimization problem with volume constraint}, SIAM J. Control and Opimization Vol. 24 No. 2 (1986), 191-198.

\bibitem{AC} H. W. Alt, L. Caffarelli, {\em Existence and regularity for a minimum problem with free boundary}, Journal fur die reine und angewandte Mathematik Vol. 325 (1981), 105-144.

\bibitem{ACS}N. Aguilera, L. Caffarelli, J. Spruck, {\em An optimization problem in heat conduction}, Annali della Scuola Normale Superiore di Pisa, Classe di Scienze 4 serie, tome 14, No 3 (1987), 355-387.

\bibitem{CS}L. Caffarelli, S. Salsa, {\em A geometric approach to free boundary problems}, Graduate Study in Mathematics Vol. 68, American Mathematical Society.

\bibitem{KS}D. Kinderlehrer, G. Stampacchia, {\em An introduction to variational inequalities and their applications}, Classics in Applied Mathematics (2000), SIAM.
\bibitem{PSU}A. Petrosyan, H. Shahgholian, N. Uraltseva, {\em Regularity of free boundaries in obstacle-type problems}, Graduate Study in Mathematics Vol. 136, American Mathematical Society.

\bibitem{T}E. V. Teixeira, {\em The nonlinear optimization problem in heat conduction}, Calc. Var. 24 (2005), 21-46.







\end{thebibliography}
\end{document}